\documentclass[12pt,oneside,onecolumn,reqno]{amsart}
\usepackage{amsmath}
\usepackage[dvips]{epsfig}
\usepackage{amsgen, amstext,amsbsy,amsopn, amsthm, amsfonts,amssymb,amscd,amsmat
	h,euscript,enumerate,url,verbatim,calc,oldgerm}
\usepackage{amsaddr}
\usepackage{bm}
\usepackage{latexsym}
\usepackage{color}
\usepackage{wrapfig}
\usepackage{graphicx}
\usepackage{epstopdf}
\usepackage{pifont}
\usepackage{setspace}
\usepackage{float}
\usepackage{filecontents}
\usepackage{blindtext}
\usepackage{csquotes}
\usepackage{cite}
\usepackage[caption = false]{subfig}
\usepackage{tikz}
\tikzset{->-/.style={decoration={
			markings,
			mark=at position #1 with {\arrow{>}}},postaction={decorate}}}
\usetikzlibrary{shapes,backgrounds,calc}
\usepackage{graphics}

\newcommand{\rt}{\rightarrow}
\newcommand{\om}{\omega}

\newcommand{\p}{\alpha}

\theoremstyle{plain}
\newtheorem{thm}{Theorem}[section]

\theoremstyle{definition}

\theoremstyle{remark}

\newtheorem{theorem}{Theorem}[section]

\newtheorem{lemma}[theorem]{Lemma}

\theoremstyle{definition}

\theoremstyle{remark}

\makeatletter

\newcommand{\Rmnum}[1]{\expandafter\@slowromancap\romannumeral #1@}
\makeatother

\begin{document}
	\title{ A new kind of $A_{\p}$-matrix for mixed graphs}
	\author{P. Sharma  \and   R. Kumar}
	\address{Department of Mathematics and Computing,\\Dr B R Ambedkar National Institute of Technology Jalandhar,\\ Punjab, India}
	\email{thakurrk@nitj.ac.in; piyushs.mc.24@nitj.ac.in}
	
	\begin{abstract}
		Nikiforov\cite{nikiforov2017merging} introduced the concept of the $A_{\p}$-matrix as a convex linear combination of a graph's adjacency matrix and its diagonal matrix of vertex degrees. In this paper, we introduce a new variant of the $A_{\p}$-matrix, which is a linear combination of the Hermitian adjacency matrix of the second kind, $H^{\om}$, and the degree-diagonal matrix. This study provides new insights into the $A_{\p}$-matrix as a Hermitian matrix, leading to novel bounds in spectral graph theory.
	\end{abstract}
	
	\thanks{$^{\dag}$ AMS classification 05C50, 05C05, 15A18}
	\maketitle



	\section{Introduction} 
	In graph theory, a mixed graph is a graph that contains both directed and undirected edges. An undirected edge can be treated as a bidirectional edge. Moreover, a mixed graph $M_{G}$ is defined as an ordered triple $M_{G} = (V,E,A)$, where $V$ denotes the set of vertices, $E$ is the set of undirected edges, and $A$ is the set of directed edges or arcs. For a mixed graph $M_{G}$, the corresponding underlying graph is the simple, undirected graph which we get by stripping away all the direction from its edges. Let $ V = \{ v_{1}, v_{2}, \dots ,v_{n}\}$ be the vertex set. The number of vertices in the graph, known as its order, is denoted by $|V|$. The degree of a vertex $v_{i}$, denoted by $d_{i}$, is the number of edges adjacent to it in the underlying graph of $M_{G}$. \\
	Throughout this paper, we consider only finite and simple mixed graphs, that is, those without loops or multiple arcs. For basic notation and terminology not defined here, we refer the reader to  \cite{bapat2010graphs, stanic2015inequalities, li2022hermitian} and the citations therein. The spectral properties of graphs have diverse applications in combinatorics and mathematical chemistry \cite{li2012graph}. Let $M_{G}$ be a mixed graph with the Hermitian adjacency matrix $H^{\beta}$ introduced by Mohar\cite{mohar2020new}, and let $G$ be the underlying graph of  $M_{G}$ with the degree diagonal matrix $D(G)$. Nikiforov\cite{nikiforov2017merging} introduced the $A_{\p}$-matrix as a linear convex combination of the adjacency matrix $A$ and the degree diagonal matrix $D$ of a graph $G$.  This matrix has recently garnered significant attention because of its unifying structure. Our work is motivated by the concept of the Hermitian adjacency matrix introduced by Mohar  \cite{mohar2020new}. In this paper, we study the spectra of a new kind of  $A_{\p}$-matrix, defined as a convex combination of the Hermitian adjacency matrix of the second kind, $H^{\om}$, and the degree diagonal matrix $D$ for mixed graphs. Mohar\cite{mohar2020new} defined the Hermitian adjacency matrix $H^{\beta}$ as 
	\begin{equation*}
		H^{\beta}_{vu} = \begin{cases}
			\beta , \text{if there is an arc from $v\rt u$ ;} \\
			\bar{\beta}, \text{if there is an arc from $u\rt v$ ;} \\
			1, \text{if $\{v,u\}$  is an undirected edge ;} \\
			0, \text{otherwise}
		\end{cases}
	\end{equation*}
	where $H^{\beta}_{vu}$ represents the $(v,u)^{th}$ entry of the Hermitian adjacency matrix $H^{\beta}$. 
	Let $\beta = a + ib$ be any complex number with modulus one (that is, $|\beta | =1 $), where $a \geq 0 \ \text{and} \  b\in \mathbb{R}$, and let $\bar{\beta}= a-ib$ be its conjugate. For any $\p \in [0,1]$, we define a new kind of $A_{\p}^{\beta}$-matrix, using the new Hermitian adjacency matrix $ H^{\beta}$:
	\begin{equation}\label{Aalphagen}
		A_{\p}^{\beta}=\p D +(1-\p) H^{\beta}.
	\end{equation}
	We are particularly interested in the case $\beta = \om $, where $\om =\frac{1+i\sqrt{3}}{2}$ is the sixth root of unity. This root is important due to its special properties:  $\om + \bar{\om} = 1$ and $\om \cdot \bar{\om} =1$ (see \cite{mohar2020new} for more details). For $\beta =\om$, we obtain the Hermitian adjacency matrix of the second kind, $H^{\om}$, and its corresponding $A_{\p}^{\om}$-matrix. For simplicity, we denote $A_{\p}^{\om}$ as $A_{\p}$, and therefore
	\begin{equation}\label{Aalpha}
		A_{\p}=\p D +(1-\p) H^{\om}.
	\end{equation}
	Let $A$ be any $(n\times n)$ complex matrix with eigenvalues  $\mu_{1}(A), \mu_{2}(A), \dots , \mu_{n}(A) $. The spectral radius of $A$, denoted  $\rho(A)$, is defined as: 
	\begin{equation*}
		\rho(A) = max \left \{ |\mu_{i}(A) | : i = 1,2, \dots ,n \right \}.
	\end{equation*}
	In this work, we first discuss some basic properties of this new kind of $A_{\p}$-matrix. We then provide upper and lower bounds for its extreme eigenvalues, as well as for its spread, and spectral radius.
	\section{Preliminaries}
	In this section, we first discuss some basic properties of the $A_{\p}$-matrix, and some results in the form of lemmas that are present in the literature. \\
	The $A_{\p}^{\beta}$-matrix, as defined in \eqref{Aalphagen}, is a Hermitian matrix, due to which it has great importance in spectral graph theory. The eigenvalues of $A_{\p}^{\beta}$-matrix are real. Let $\beta_{1}(A_{\p}^{\beta}) \geq \beta_{2}(A_{\p}^{\beta}) \geq \dots \geq \beta_{n}(A_{\p}^{\beta})$ be the eigenvalues of the $A_{\p}^{\beta}$-matrix, arranged in non-increasing order. For the particular case where $\beta = \om $,  the eigenvalues are arranged in non-increasing order as   $\mu_{1}(A_{\p}) \geq \mu_{2}(A_{\p}) \geq \dots \geq \mu_{n}(A_{\p})$.
	\begin{lemma}\label{prop1}
		For any mixed graph $M_{G}$ with vertex set $V$, 	$A_{\p}^{\beta}$-matrix has the following properties:
		\begin{enumerate}
		
	\item[(i)] All the eigenvalues of $A_{\p}^{\beta}$-matrix are real and it has $|V|$ orthonormal eigenvectors.
			\item[(ii)]	$A_{\p}^{\beta}$-matrix is unitarily similar to a diagonal matrix.
			\item[(iii)] The numerical range of 	$A_{\p}^{\beta}$ matrix, that is, $R = \{ z^*A_{\p}^{\beta}z \ | \ z \in \mathbb{C}^V, ||z||=1 \}$ is an interval of the real axis and $\beta_{n}(A_{\p}^{\beta}) = \text{min} \ (R)$ and $\beta_{1}(A_{\p}^{\beta}) = \text{max} \ (R)$.
		\end{enumerate}
	\end{lemma} 
	The min-max formula for $\beta_{j}(M_{G})$ is a direct consequence of Lemma \ref{prop1}:
	\begin{equation*}
		\beta_{j}(M_{G}) = \max_{\text{dim}\ U = j} \ \min_{\substack{z \in U \\ ||z||=1}} z^* A_{\p}^{\beta} z = \min_{\text{dim}\  U = n-j+1} \max_{\substack{z \in U \\ ||z||=1}} z^* A_{\p}^{\beta} z,
	\end{equation*}
	where $U$ is any subspace of $\mathbb{C}^{|V|}$ of dimension $j$, and $n-j+1$. \\
	Let $z = x + iy$, where $x,y \in \mathbb{R}^{|V|}$ and $\beta = a + ib$, we have
	
	\begin{align}
		<A_{\p}^{\beta}z,z> &= z^*A_{\p}^{\beta}z = \p z^* Dz + (1-\p) z^*H^{\beta}z  \nonumber \\
		&= \p \sum_{i=1}^{n}d_i |z_i|^2 + (1-\p) \sum_{v\in V}\bar{z_v}\sum_{u\in V}H^{\beta}_{vu}z_u  \nonumber 
	\end{align}
	and therefore,
	\begin{equation}\label{quadratic1}
		z^*A_{\p}^{\beta}z= \p \sum_{i=1}^{n}d_i|z_i|^2+ (1-\p)\sum_{v\rt u}\bar{z}_v z_u \beta + \bar{z}_u z_v \bar{\beta}.
	\end{equation} 
	Also, we have
	\begin{align}
		<A_{\p}^{\beta}z,z>  &= \p \sum_{i=1}^{n}d_i |z_i|^2 + (1-\p) \sum_{v\in V}\bar{z_v}\sum_{u\in V}H^{\beta}_{vu}z_u  \nonumber \\
		&= \p \sum_{i=1}^{n}d_i |z_i|^2 + (1-\p)\sum_{v\in V}(x_v - iy_v)\sum_{u\in V}H^{\beta}_{vu}(x_u + iy_u)  \nonumber \\ 
		&=	\p \sum_{i=1}^{n}d_i |z_i|^2 + (1-\p) \sum_{v\rt u}\left [ (x_v - iy_v)(x_u + iy_u)\beta + (x_u -iy_u)(x_v + iy_v)\bar{\beta} \right ]  \nonumber 
	\end{align}  
	and hence,
	\begin{equation}\label{quadratic2}
		z^*A_{\p}z= \p \sum_{i=1}^{n}d_{i} |z_{i}|^2 + (1-\p) \sum_{v\rt u}(2ax_vx_u + 2a y_vy_u -2bx_vy_u+ 2b y_vx_u).
	\end{equation}
	In this quadratic form $z^*A_{\p}^{\beta}z$, the summation runs over all arcs $e = vu$.
	For $\beta =\om $, we get the following expression:
	\begin{align}
		< A_{\p}z,z> &= z^* A_{\p}z \nonumber \\
		&= \p \sum_{i=1}^{n}d_i |z_i|^2 + (1-\p) \sum_{v\rt u}\bar{z_v}z_u \om + \bar{z_u}z_v \bar{\om} \nonumber \\
		&= \p \sum_{i=1}^{n}d_i |z_i|^2 + (1-\p) \sum_{v\rt u}(x_vx_u + y_vy_u - \sqrt{3} x_vy_u + \sqrt{3}y_vx_u). \label{quadAalpha}
	\end{align}

	\begin{lemma}
		Let $M_{G}$ be a mixed graph with $n$ vertices, $m$ edges, $a$ arcs, and $u$ undirected edges. Then
		\begin{equation*}
			\mu_{1}(A_{\p}) \geq \frac{2\p m + (1-\p)(a + 2u)}{n}.
		\end{equation*}
	\end{lemma}
	\begin{proof}
		Let $z=x$ be the constant real vector with coordinates $x_{v}= \frac{1}{\sqrt{n}}$, so $||x||=1$. From the condition $y=0$ and equation \eqref{quadAalpha}, we have 
		\begin{align}
			x^{*} A_{\p} x &= \frac{\p \sum_{i=1}^{n}d_{i}}{n} + \frac{(1-\p) (a + 2u)}{n}   \nonumber \\
			&= \frac{\p (2m) + (1-\p) (a + 2u)}{n}. \nonumber
		\end{align}
		This implies that 
		\begin{equation*}
			\mu_{1}(A_{\p}) \geq \frac{2\p m + (1-\p)(a + 2u)}{n}.
		\end{equation*}
	\end{proof}

	\begin{lemma}\cite{li2022hermitian}\label{trlemma}
		Let $M_{G}$ be a mixed graph with $n$ vertices, $m$ edges, and let $H^{\om}$ be its Hermitian matrix of the second kind. If $\mu_{1}(H^{\om}), \mu_{2}(H^{\om}), \dots , \mu_{n}(H^{\om})$ are the eigenvalues of $H^{\om}$, then 
		\begin{equation*}
			\sum_{i=1}^{n}\mu_{i}^2(H^{\om}) = \text{tr} \left ( (H^{\om})^2 \right ) = \sum_{i=1}^{n}d_{i} = 2m.
		\end{equation*}
		
	\end{lemma}
	
	\begin{lemma}
		Let $A_{\p}$ be a matrix for any mixed graph $M_{G}$ with a vertex set $V \ (\text{where}\ |V|=n)$, let $G$ be the underlying graph of $M_{G}$ with $m$ edges, and let $d_{1}, d_{2}, \dots ,d_{n}$ be the degrees of the vertices $v_{1}, v_{2}, \dots , v_{n}$, respectively. Then 
		\begin{equation}\label{trAalpha}
			\text{tr}\ (A_{\p}) = 2\p m
		\end{equation}
		and
		\begin{equation}\label{trA2alpha}
			\text{tr}\ (A_{\p}^2) = \p ^2 \sum_{i=1}^{n}d_i^2+ (1-\p)^2 2m.
		\end{equation}
	\end{lemma}
	\begin{proof}
		We observe that $\text{tr}\ (H^{\om})=0$, and from Lemma \ref{trlemma}, we have $ \text{tr} \left ( (H^{\om})^2 \right ) = \sum_{i=1}^{n}d_{i} = 2m $. A simple calculation then leads to \eqref{trAalpha} and \eqref{trA2alpha}, that is,
		\begin{equation*}
			\text{tr}\ (A_{\p}) = \p \text{tr}\ (D) + (1-\p) \text{tr}\ (H^{\om}) = 2 \p m    
		\end{equation*}
		and 
		\begin{align}
			\text{tr}\ (A_{\p}^2) &= \text{tr} \left (\p^2D^2 + \p (1-\p) DH^{\om} + \p (1-\p) H^{\om}D + (1-\p)^2 (H^{\om})^2 \right )  \nonumber \\ 
			&= \p^2 \text{tr} (D)^2 + \p (1-\p) \text{tr}(DH^{\om}) + \p (1-\p) \text{tr}(H^{\om}D) + (1-\p)^2 \text{tr} \left ( (H^{\om})^2\right ) \nonumber \\
			&= \p ^2 \sum_{i=1}^{n}d_i^2+ (1-\p)^2 2m . \nonumber
		\end{align}
	\end{proof}
	We need the following lemmas from the literature to prove our main results:
	\begin{lemma}\cite{garga2015inequalities}\label{th.sirI}
		Let $A= [a_{rs}]$ be any Hermitian matrix of order $n \times n $ and let the eigenvalues of $A$ be ordered as $\mu_{max}(A) = \mu_{1}(A) \geq \mu_{2}(A) \geq \dots \geq \mu_{n}(A) = \mu_{min}(A)$. Then 
		\begin{equation*}
			\mu_{min}(A) + \frac{2}{n} |a_{rs} | \leq \frac{\text{tr}\left (A\right )}{n} \leq \mu_{max}(A) - \frac{2}{n} |a_{rs}|,
		\end{equation*}
		hold for $r\neq s$.
	\end{lemma}

	\begin{lemma}\label{wolkowiczbounds} \cite{wolkowicz1980bounds}
		Let $A$ be an $n\times n$ complex matrix with real eigenvalues arranged in decreasing order as $\mu_{max}(A) = \mu_{1}(A) \geq \mu_{2}(A) \geq \dots \geq \mu_{n}(A) = \mu_{min}(A)$, and let $r=\frac{\text{tr}\ \left (A\right )}{n}$, $s^2=\frac{\text{tr}\ \left (A^2\right )}{n}-r^2$. Then 
		\begin{equation*}
			r-s\sqrt{n-1} \leq \mu_{min}(A) \leq r- \frac{s}{\sqrt{n-1}}  
		\end{equation*}
		and
		\begin{equation*}
			r + \frac{s}{\sqrt{n-1}} \leq \mu_{max}(A) \leq 	r + s\sqrt{n-1}.
		\end{equation*}
		Moreover, for any $j = 1, 2, \dots , n$, we have 
		\begin{equation*}
			r - s\sqrt{\frac{j-1}{n-j+1}} \leq \mu_{j}(A) \leq r + s\sqrt{\frac{n-j}{j}}.
		\end{equation*}
	\end{lemma}
	
	\begin{lemma}\label{spreadbounds} \cite{wolkowicz1980bounds}
		Let $A$, $r$, and $s^{2}$ be defined as in Lemma \ref{wolkowiczbounds}. Then
		\begin{equation}\label{B.S.upper}
			\mu_1(A) - \mu_n(A) \leq (2n)^{\frac{1}{2}}s.
		\end{equation}
		If $n=2q$ is even, then
		\begin{equation}\label{evencaseL.S.B.}
			2s \leq 	\mu_{1}(A) - \mu_{n}(A).
		\end{equation}
		If $n=2q \pm 1$ is odd, then \eqref{evencaseL.S.B.} holds, but moreover, \begin{equation}\label{oddcaseL.S.B.}
			\frac{2sn}{\sqrt{n^2-1}} \leq 	\mu_{1}(A) - \mu_{n}(A).
		\end{equation}
		
	\end{lemma}
	
	The first Zagreb index of a graph is defined in terms of the degrees of its vertices, given by the formula $M_1(G) = \sum_{i=1}^{n}d_i^2$. For more details, see \cite{gutman1972graph}. The following lower bound for $M_1(G)$ is obtained by Mamta and Kumar \cite{verma2025new}:

	\begin{lemma}\cite{verma2025new}\label{mverma}
		Let $G$ be a simple graph with $n\geq 3$ vertices and $m$ edges. Then
		\begin{equation}\label{zagribverma}
			M_1(G) \geq \frac{4m^2}{n} + \frac{1}{2} \left (\Delta - \delta \right )^2 + \frac{2n}{n-2} \left(\frac{2m}{n} - \frac{\Delta + \delta}{2}\right)^2.
		\end{equation}
	\end{lemma}

	\section{Main results}
	In this section, we present bounds for the extreme eigenvalues and the spectral radius of the $A_{\p}$-matrix. Moreover, we also derive an upper bound for the trace norm of the $A_{\p}$-matrix, as well as bounds for its spread.
	
	\begin{thm}\label{firstthm}
		Let $A_{\p}$ be a matrix for any mixed graph $M_{G}$ with vertex set $V$ (where $|V| = n$), and let $G$ be the underlying graph of $M_{G}$ with $m$ edges. Let $\mu_{max}(A_{\p}) = \mu_{1}(A_{\p}) \geq \mu_{2}(A_{\p}) \geq \dots \geq \mu_{n}(A_{\p}) = \mu_{min}(A_{\p})$ be the eigenvalues of $A_{\p}$ arranged in decreasing order. Then 
		\begin{equation*}
			\mu_{max}(A_{\p}) \geq \frac{2(\p m +1)}{n}  \nonumber 
		\end{equation*}
		and
		\begin{equation*}
			\mu_{min}(A_{\p}) \leq \frac{2(\p m - 1)}{n}. \nonumber
		\end{equation*}
		
	\end{thm}
	\begin{proof}
		Let $A_{\p} = [a_{rs}]$ be an $n \times n$ matrix. Note that for any non-zero element $a_{rs}$ of $A_{\p}$, we have $|a_{rs}| = 1$. Therefore, using Lemma \ref{th.sirI}, we find that \begin{equation*}
			\mu_{max}(A_{\p}) \geq \frac{2\p m}{n} + \frac{2}{n} = \frac{2(\p m+1)}{n}  
		\end{equation*}
		and
		\begin{equation*}
			\mu_{min}(A_{\p}) \leq =  \frac{2\p m}{n} - \frac{2}{n} = \frac{2(\p m -1)}{n}.  
		\end{equation*} 
	\end{proof}
	Note that Theorem \ref{firstthm} is also valid for the more general $A_{\p}^{\beta}$-matrix, since $|\beta| = 1$.
	
	\begin{thm}
		Let $M_{G}$ be a mixed graph with vertex set $V$ and edge set $E$. Let $A_{\p}^{\beta}$ be the $\p$-adjacency matrix as defined in \eqref{Aalphagen}, where $\beta = a + ib$ with $a, b \in \mathbb{R}$, $|\beta| = 1$, and $a \geq 0$. Let the eigenvalues of $A_{\p}^{\beta}$ be ordered as $\beta_{1}(A_{\p}^{\beta}) \geq \beta_{2}(A_{\p}^{\beta}) \geq \dots \geq \beta_{n}(A_{\p}^{\beta})$. Then
		\begin{equation*}
			\frac{1}{3}\rho (A_{\p}^{\beta})\leq \beta _{1}(A_{\p}^{\beta})\leq \rho (A_{\p}^{\beta}).
		\end{equation*}
	\end{thm}
	\begin{proof}
		The second inequality trivially holds, so we only have to prove the first one. If $\rho (A_{\p}^{\beta}) \neq \beta _{1}(A_{\p}^{\beta})$, then $\rho = \rho (A _{\p}^{\beta})=|\beta _{n}(A_{\p}^{\beta})|$. Let $z=x+iy$ be a unit eigenvector of $\beta _{n}(A_{\p}^{\beta}).$ Without loss of generality we may assume that $x\neq 0$ and $y\neq 0$. From \eqref{quadratic2}, we have  $-\rho = \beta _{n}(A_{\p}^{\beta})=z^*A_{\p}^{\beta}z=P+Q+R+S$, where	$P=(1-\p)\sum_{v\rt u}2ax_v x_u, Q=(1-\p)2a\sum_{v\rt u}y_v y_u, R=(1-\p)2b \sum _{v\rt u}(y_vx_u - x_v y_u),$ and $S= \p \sum_{i=1}^{n}d_i|z_i|^2.$
		Suppose first that $|P| + |Q| \geq \tfrac{\rho}{3}$. Let $\hat{x}, \hat{y} \in \mathbb{R}^{|V|}$ be real vectors whose coordinates are $|x_v|$ and $|y_v|$, respectively.
		Then $\hat{x}^*A^{\beta}_{\p}\hat{x}\geq |P|$ and $\hat{y}^*A^{\beta}_{\p}\hat{y}\geq |Q|$. Since $||\hat{x}||^2 + ||\hat{y}||^2=||z||^2=1$, it follows that $\tfrac{\hat{x}^*A^{\beta}_{\p}\hat{x}}{||\hat{x}||^2} \geq |P| +|Q|$ or $\tfrac{\hat{y}^*A^{\beta}_{\p}\hat{y}}{||\hat{y}||^2} \geq |P| +|Q|$. In either case, it follows that  $\beta _1(A_{\p}^{\beta}) \geq |P| + |Q| \geq \tfrac{\rho}{3}$. \\ 
		Now, suppose that $|P| + |Q| \leq \tfrac{\rho}{3}$, since $\rho = -P-Q-R-S \leq |P| + |Q| +|S| - R \leq \tfrac{\rho}{3} +S-R$, we conclude that $
			\frac{2}{3}\rho \leq S-R.$
		Let us consider $\bar{z}=x -iy$, it follows that $\beta _1(A_{\p}^{\beta}) \geq \bar{z}^*A_{\p}^{\beta}\bar{z}= P+Q+S-R$. \\
		Therefore, 
		$\rho = -P-Q-R-S \leq -P-Q+S-R \leq -P - Q + \mu_1 -P-Q \leq 2(|P| + |Q|)+ \mu _1$, which implies that $ \rho = \rho (A _{\p}^{\beta}) \leq 3\beta _1(A_{\p}^{\beta})$.
	\end{proof}
	
	\begin{thm}\label{boundspdrad2}
		Let the $A_{\p}$-matrix be defined as in \eqref{Aalpha}. Let the eigenvalues of $A_{\p}$ be ordered as $\mu_{1}(A_{\p}) \geq \mu_{2}(A_{\p}) \geq \dots \geq \mu_{n}(A_{\p})$. Then
		\begin{equation}
			\frac{1}{2}\rho (A_{\p}) \leq \mu_1 (A_{\p}) \leq \rho (A_{\p}).
		\end{equation}
	\end{thm}
	\begin{proof}
		Assume $\rho(A_{\p}) =-\mu_n(A_{\p}) > \mu _1(A_{\p})$. From \eqref{quadratic1}, $z^*A_{\p}^*z= \p \sum_{i=1}^{n}d_i|z_i|^2+ (1-\p)\sum_{v\rt u}\bar{z}_v z_u \om + \bar{z}_u z_v \bar{\om}$. Let $z$ be unit eigenvector for $\mu_n(A_{\p})$ and let $\hat{z}\in \mathbb{R}^{|V|}$ be the vector with entries $\hat{z}_v=|z_v|$, $v\in V.$ Let us observe that $||\hat{z}||=1$ and $$\hat{z}_v\hat{z}_u = |\om| |z_v||z_u|=|\om \bar{z}_v z_u| = \frac{1}{2}|\om \bar{z}_vz_u| + \frac{1}{2}|\bar{\om}\bar{z}_uz_v| \geq \frac{1}{2}|\om \bar{z}_vz_u + \bar{\om}\bar{z}_uz_v|$$.
		By using this inequality and \eqref{quadratic1}, we have 
		\begin{align}
			\frac{1}{2}\rho(A_{\p}) &= \frac{1}{2}|z^*A_{\p}z| \nonumber \\
			&=\frac{1}{2}|\p \sum_{i=1}^{n}d_i |z_i|^2 +(1-\p)\sum_{v\rt u}\left ( \bar{z}_vz_u\om + \bar{z}_u z_v\bar{\om}\right )|  \nonumber \\
			&\leq \frac{1}{2}|\p \sum_{i=1}^{n}d_i|z_i|^2| + \frac{1-\p}{2}| \sum_{v\rt u} \left (\bar{z}_v z_u \om + \bar{z}_u z_v\bar{\om}\right )| \nonumber  \\
			&\leq \p \sum_{i=1}^{n}d_i|z_i|^2 + \frac{1-\p}{2}\sum_{v\rt u} \left (|\bar{z}_v||z_u| + |\bar{z}_u| |z_v|\right )   \nonumber \\
			&= \p \sum_{i=1}^{n}d_i|z_i|^2 + (1-\p)\sum_{v\rt u}|z_u||z_v| \nonumber  \\
			&= \hat{z}^*A_{\p}\hat{z}, \nonumber
		\end{align}
		which implies that 
		\begin{equation*}
			\frac{1}{2}\rho(A_{\p}) \leq  \mu_1(A_{\p}).
		\end{equation*}
	\end{proof}

	\begin{thm}
		Let $M_{G}$ be a mixed graph with n vertices, and let $G$ be its underlying graph with $m$ edges, a maximum degree $\Delta $, and a minimum degree $\delta$. Then for any $\p \in [0,1]$, we have
		\begin{equation}\label{thm3.4eq1}
			\mu_{min}(A_{\p}) \leq \frac{2 \p m}{n} - \sqrt{\frac{\frac{n\p^2}{2} \left (\Delta - \delta \right )^2 + \frac{2n^2\p^2}{n-2} \left (\frac{2m}{n}- \frac{\Delta +\delta}{2} \right )^2 + n \left (1-\p \right )^2 2m}{n^2(n-1)}}  
		\end{equation}
		and
		\begin{equation}\label{thm3.4eq2}
			\mu_{max}(A_{\p}) \geq \frac{2\p m}{n} + \sqrt{\frac{\frac{n\p^2}{2} \left (\Delta - \delta \right )^2 + \frac{2n^2\p^2}{n-2} \left (\frac{2m}{n} - \frac{\Delta +\delta}{2} \right )^2 + \left (1-\p \right )^2 2mn}{n^2(n-1)}}.
		\end{equation}
	\end{thm}
	\begin{proof}
		From \eqref{trAalpha}, \eqref{trA2alpha} and Lemma \ref{wolkowiczbounds}, we find that
		\begin{align}
			\mu_{min}(A_{\p}) &\leq \frac{\text{tr}(A_{\p})}{n} - \sqrt{ \frac{1}{n-1} \left [ \frac{\text{tr}(A^2_{\p})}{n} - \left (\frac{\text{tr}(A_{\p})}{n} \right )^2 \right ] }    \nonumber \\ 
			&= \frac{2\p m}{n} - \sqrt{\frac{n\p^2\sum_{i=1}^{n}d_i^2 + n \left (1-\p \right )^2 2m - 4\p^2m^2}{n^2(n-1)}}.   \label{thm3.4eq3}
		\end{align}
	The inequality \eqref{thm3.4eq1} follows on combining \eqref{zagribverma} and \eqref{thm3.4eq3}.
		Likewise, \eqref{thm3.4eq2} follows on combining \eqref{trAalpha}, \eqref{trA2alpha}, Lemma \ref{wolkowiczbounds} and Lemma \ref{mverma}.
	\end{proof}
	\begin{thm}
		Let $M_{G}$ be a mixed graph with n vertices, and let $G$ be its underlying graph with $m$ edges, a maximum degree $\Delta $, and a minimum degree $\delta$. Then for any $\p \in [0,1]$, we have
		\begin{equation}\label{thm3.4eq4}
			\mu_{max}(A_{\p}) \leq \frac{2\p m + \sqrt{(n-1) \left [ n\p^2\sum_{i=1}^{n}d_i^2 + \left (1-\p \right )^2 2mn- 4\p^2m^2 \right ] }}{n}
		\end{equation}
		and
		\begin{equation}\label{thm3.4eq5}
			\mu_{min}(A_{\p}) \geq \frac{2\p m - \sqrt{(n-1) \left [ n\p^2\sum_{i=1}^{n}d_i^2 + \left (1-\p \right )^2 2mn - 4\p^2m^2 \right ] }}{n}.    
		\end{equation}
	\end{thm}
	\begin{proof}
	The inequality \eqref{thm3.4eq4} follows on combining Lemma \ref{wolkowiczbounds}, \eqref{trAalpha} and \eqref{trA2alpha}, that is, \\
		\begin{align}
			\mu_{max}(A_{\p}) &\leq \frac{\text{tr}(A_{\p})}{n} + \sqrt{(n-1) \left [ \frac{\text{tr}A^2_{\p}}{n} - \left ( \frac{\text{tr}(A_{\p})}{n} \right )^2 \right ] }   \nonumber  \\
			&= \frac{2\p m + \sqrt{(n-1) \left [ n\p^2\sum_{i=1}^{n}d_i^2 + \left (1-\p \right )^2 2mn - 4\p^2 m^2 \right ] }}{n}.  \nonumber 
		\end{align} 
	Likewise, \eqref{thm3.4eq5} follows on combining Lemma \ref{wolkowiczbounds}, \eqref{trAalpha} and \eqref{trA2alpha}.
	\end{proof}
	
	If $\mu_{1}(A_{\p}), \mu_{2}(A_{\p}), \dots, \mu_{n}(A_{\p})$ are the eigenvalues of the $A_{\p}$-matrix for $\p \in [0,1]$, then its trace norm (denoted by $||A_{\p}||_{*}$), also known as the graph energy, is defined \cite{bhat2025bounds} as the sum of the absolute values of the eigenvalues; that is,
	
	\begin{equation*}
		||A_{\p}||_{*} = \sum_{i=1 }^{n}|\mu_{i}(A_{\p})|.
	\end{equation*}
	For more on graph energy, see \cite{balakrishnan2004energy, nikiforov2016beyond, rada2010lower, ganie2023increasing, das2018degree} and the references therein. Now, we present an upper bound on the trace norm of $A_{\p}$:
	\begin{thm}
		Let $M_{G}$ be a mixed graph with $n$ vertices, and let $G$ be its underlying graph with $m$ edges. Let $\mu_{1}(A_{\p}) \geq \mu_{2}(A_{\p}) \geq \dots \geq \mu_{n}(A_{\p})$ be the eigenvalues of $A_{\p}$, where $\p \in [0,1]$. Then
		\begin{equation}\label{thm3.6m}
			||A_{\p}||_{*}  \leq  4\p m + 2 \sqrt{(n-1) \left [ n\p^2\sum_{i=1}^{n}d_i^2 + \left (1-\p \right )^2 2mn - 4\p^2m^2 \right] }.
		\end{equation}
	\end{thm}
	\begin{proof}
		Observe that
		\begin{align}
			||A_{\p}||_{*}&= \sum_{i=1 }^{n}|\mu_{i}(A_{\p})|  \nonumber  \\
			&\leq \sum_{i=1 }^{n}\rho (A_{\p}) = n \rho (A_{\p}).  \nonumber
		\end{align}
		Hence, using Theorem \ref{boundspdrad2}, we get
		\begin{align}
			||A_{\p}||_{*}&\leq 2n\mu_{1}(A_{\p})  \nonumber \\
			&\leq 2n \left [ \frac{2\p m + \sqrt{(n-1) \left [ n\p^2 \sum_{i=1 }^{n}d_i^2 + \left (1-\p \right )^2 2mn - 4\p^2m^2 \right ] }}{n} \right ]  \nonumber \\ 
			&= 4\p m + 2 \sqrt{(n-1) \left [ n\p^2\sum_{i=1}^{n}d_i^2 + \left (1-\p \right )^2 2mn - 4\p^2m^2 \right] },  \nonumber
		\end{align}which proves \eqref{thm3.6m}.
	\end{proof}
	
	\begin{thm}
		Let $M_{G}$ be a mixed graph with $n$ vertices, and let $G$ be its underlying graph with $m$ edges. Let $\mu_{1}(A_{\p}) \geq \mu_{2}(A_{\p}) \geq \dots \geq \mu_{n}(A_{\p})$ be the eigenvalues of $A_{\p}$, where $\p \in [0,1]$. Then, for any $j \in \{1,2, \dots, n\}$, we have
		
		\begin{equation*}
			\mu_{j}(A_{\p}) \leq \frac{2\p m}{n} + \sqrt{\frac{n-j}{j} \left [ \frac{n\p^2\sum_{i=1}^{n}d_{i}^2 + \left (1-\p \right )^2 2mn -4\p^2m^2}{n^2} \right ]}
		\end{equation*}
		and 
		\begin{equation*}
			\mu_{j}(A_{\p}) \geq \frac{2\p m}{n} - \sqrt{\frac{j-1}{n-j+1} \left [ \frac{n\p^2\sum_{i=1}^{n}d_i^2 + \left (1-\p \right )^2 2mn - 4\p^2 m^2}{n^2} \right ]}.
		\end{equation*}
	\end{thm} 
	\begin{proof}
		From \eqref{trAalpha},\eqref{trA2alpha} and  Lemma \ref{wolkowiczbounds}, we obtain 
		\begin{align}
			\mu_{j}(A_{\p}) &\leq \frac{\text{tr}(A_{\p})}{n} + \sqrt{\frac{n-j}{j}\left [ \frac{\text{tr}(A_{\p}^2)}{n} - \left (\frac{\text{tr}(A_{\p})}{n} \right )^2  \right ]}   \nonumber \\
			&= \frac{2\p m}{n} + \sqrt{\frac{n-j}{j} \left [ \frac{n\p^2\sum_{i=1}^{n}d_{i}^2 + \left (1-\p \right )^2 2mn -4\p^2m^2}{n^2} \right ]}.  \nonumber
		\end{align}
		Similarly, we find
		\begin{align}
			\mu_{j}(A_{\p}) &\geq  \frac{\text{tr}(A_{\p})}{n} - \sqrt{\frac{j-1}{n-j+1} \left [ \frac{\text{tr}(A_{\p}^2)}{n} - \left (\frac{\text{tr}(A_{\p})}{n} \right )^2 \right ]}  \nonumber \\
			&= \frac{2\p m}{n} - \sqrt{\frac{j-1}{n-j+1} \left [ \frac{n\p^2\sum_{i=1}^{n}d_i^2 + \left (1-\p \right )^2 2mn - 4\p^2 m^2}{n^2} \right ]}.  \nonumber
		\end{align}This completes the proof.
	\end{proof}
	The concept of the spread of a matrix is due to Mirsky \cite{mirsky1956spread}. Let $A_{\p}$ be a matrix corresponding to any mixed graph $M_{G}$, and let $\mu_{1}(A_{\p}) \geq \mu_{2}(A_{\p}) \geq \dots \geq \mu_{n}(A_{\p})$ be its eigenvalues. Then the spread of the matrix $A_{\p}$, denoted by $\mathrm{spd}\ (A_{\p})$, is defined as $\mathrm{spd}\ (A_{\p}) = \mu_{1}(A_{\p}) - \mu_{n}(A_{\p})$. Now, we present bounds for the spread of the $A_{\p}$-matrix:
	\begin{thm}
		Let $M_{G}$ be a mixed graph with $n$ vertices, and let $G$ be its underlying graph with $m$ edges, maximum degree $\Delta$, and minimum degree $\delta$. Then for $\p \in [0,1]$, 
		\begin{equation*}
			\text{spd} \ (A_{\p}) \leq \sqrt{\frac{2n \p^2 \sum_{i=1}^{n}d_i^2 + \left (1 - \p \right )^2 4mn - 8\p^2m^2}{n}}.
		\end{equation*}
		Furthermore, if $n = 2q$ is even, then 
		\begin{equation}\label{evencase}
			\text{spd} \ (A_{\p}) \geq \frac{2}{n}\sqrt{\frac{n \p^2}{2}(\Delta - \delta )^2 + \frac{2n^2 \p^2}{n-2} \left (\frac{2m}{n} - \frac{\Delta + \delta}{2} \right )^2 + \left (1 - \p \right )^2 2mn },
		\end{equation}
	and	if $n = 2q\pm 1$ is odd, then  
		\begin{equation}\label{oddcase}
			\text{spd} \ (A_{\p}) \geq 2 \sqrt{\frac{\frac{n \p^2 }{2}(\Delta - \delta )^2 + \frac{2n^2\p^2}{n-2} \left (\frac{2m}{n} - \frac{\Delta + \delta}{2} \right )^2 + \left (1-\p \right )^2 2mn}{n^2 - 1}}.
		\end{equation}
	\end{thm}
	\begin{proof}
		From \eqref{B.S.upper} of Lemma \ref{spreadbounds}, we have 
		\begin{align}
			\text{spd} \ (A_{\p}) &\leq \sqrt{2n \left [ \frac{\text{tr}(A_{\p}^2)}{n} - \left (\frac{\text{tr}(A_{\p})}{n} \right )^2 \right ]}  \nonumber \\ 
			&= \sqrt{\frac{2n\p^2 \sum_{i=1}^{n}d_i^2 + \left (1- \p \right )^2 4mn - 8\p^2m^2}{n}}.  \nonumber
		\end{align}
		If $n = 2q $ is even, then from \eqref{evencaseL.S.B.} of Lemma \ref{spreadbounds}, we have 
		\begin{align}
			\text{spd} \ (A_{\p}) &\geq 2 \sqrt{\frac{\text{tr}(A_{\p}^2)}{n} - \left (\frac{\text{tr}(A_{\p})}{n} \right )^2}  \nonumber \\
			&= \frac{2}{n} \sqrt{n\p^2 \sum_{i=1}^{n}d_i^2 + \left (1-\p \right )^2 2mn - 4\p^2m^2},  \nonumber
		\end{align}
		and combining it with Lemma \ref{mverma}, we immediately get
		\eqref{evencase}. \\ 
	Likewise, by \eqref{oddcaseL.S.B.} of Lemma \ref{spreadbounds}, we have
		\begin{align}
			\text{spd} \ (A_{\p}) &\geq \frac{2n}{\sqrt{n^2 -1}} \sqrt{\frac{\text{tr}(A_{\p}^2)}{n} - \left (\frac{\text{tr}(A_{\p})}{n} \right )^2}  \nonumber \\
			&= 2 \sqrt{\frac{n\p^2 \sum_{i=1}^{n}d_i^2 + \left (1-\p \right )^2 2mn - 4\p^2m^2}{n^2 -1}},  \nonumber
		\end{align}
		and combining it with Lemma \ref{mverma}, we get
		\eqref{oddcase}
		
	\end{proof}

	\section*{Acknowledgments}
	The first author’s research is supported by the Ministry of Education, Government of India. The second author gratefully acknowledges the support from the National Board for Higher Mathematics (NBHM), Department of Atomic Energy (DAE), Government of India (No. 02011/30/2025/NBHM(R.P.)/R\&D-II/9676).
	\section*{Declarations}
	\subsection*{Funding}
	The authors received no funding for the preparation of this paper beside grant mentioned above.
	\subsection*{Conflict of interest} The authors have no conflicts of interest associated with this publication.
	\subsection*{Data availability} No data associated with this publication.

	
\end{document}